\documentclass[12pt]{article}

\usepackage{times}

\usepackage{enumitem}
\usepackage[utf8]{inputenc}
\usepackage{commath}
\usepackage{amsthm, amscd}
\usepackage{amsmath}
\usepackage{amssymb}
\usepackage{cite}
\usepackage{setspace}
\usepackage{ stmaryrd }
\usepackage{tikz-cd}

\usepackage{tocloft}
\usepackage{sectsty}
\usepackage{xparse}

\newcommand*\tasklabelformat[1]{#1)}

\usepackage{titlesec}

\usepackage{tasks}[newest]
\settasks{
  label = \alph* ,
  label-format = \tasklabelformat ,
  label-width  = 12pt
}
\usepackage{bigints}
\usepackage{comment}
\usepackage{mathtools}
\usepackage{mathrsfs}
\usepackage{fancyhdr}

\allowdisplaybreaks[4]

\numberwithin{equation}{section}
\usepackage{etoolbox}
\patchcmd{\thebibliography}
  {\settowidth}
  {\setlength{\itemsep}{0pt plus -10pt}\settowidth}
  {}{}
\apptocmd{\thebibliography}
  {
  }
  {}{}
\makeatletter
\newtheorem*{rep@theorem}{\rep@title}
\newcommand{\newreptheorem}[2]{%
\newenvironment{rep#1}[1]{%
 \def\rep@title{#2 \ref{##1}}%
 \begin{rep@theorem}}%
 {\end{rep@theorem}}}
\makeatother

\theoremstyle{theorem}

\newreptheorem{theorem}{Theorem}
\newtheorem{thm}{Theorem}[section]
\newtheorem*{thm*}{Theorem}
\theoremstyle{definition}
\newtheorem{prop}[thm]{Proposition}
\newtheorem*{prop*}{Proposition}

\newtheorem{lem}[thm]{Lemma}

\newtheorem*{cor*}{Corollary}
\theoremstyle{remark}

\newtheorem{rem}[thm]{Remark}

\title{\vspace*{-1.5cm} Lower bounds on fibered Yang-Mills functionals:
\\
generic nefness and semistability of direct images
}

\author
{Siarhei Finski
}

\date{}

\usepackage[%
    left=1in,%
    right=1in,%
    top=1.1in,%
    bottom=0.8in,%
    paperheight=11in,%
    paperwidth=8.5in%
]{geometry}

\newcommand{\imun} {\sqrt{-1}}

\newcommand{\sym}{{\rm{Sym}}}

\newcommand{\comp}{\mathbb{C}}
\newcommand{\real}{\mathbb{R}}

\newcommand{\nat}{\mathbb{N}}
\newcommand{\integ}{\mathbb{Z}}

\newcommand{\enmr}[1]{\text{End}{(#1)}}

\newcommand{\ccal}{\mathscr{C}}

\newcommand{\dbar}{ \overline{\partial} }

\newcommand{\rk}[1]{{\rm{rk}} ( #1 )}
\newcommand{\tr}[1]{{\rm{Tr}} \big[ #1 \big]}

\newcommand{\scal}[2]{\langle #1, #2 \rangle}

\newcommand{\td}{{\rm{Td}}}
\newcommand{\ch}{{\rm{ch}}}

\makeatletter
\DeclareFontFamily{OMX}{MnSymbolE}{}
\DeclareSymbolFont{MnLargeSymbols}{OMX}{MnSymbolE}{m}{n}
\SetSymbolFont{MnLargeSymbols}{bold}{OMX}{MnSymbolE}{b}{n}
\DeclareFontShape{OMX}{MnSymbolE}{m}{n}{
    <-6>  MnSymbolE5
   <6-7>  MnSymbolE6
   <7-8>  MnSymbolE7
   <8-9>  MnSymbolE8
   <9-10> MnSymbolE9
  <10-12> MnSymbolE10
  <12->   MnSymbolE12
}{}
\DeclareFontShape{OMX}{MnSymbolE}{b}{n}{
    <-6>  MnSymbolE-Bold5
   <6-7>  MnSymbolE-Bold6
   <7-8>  MnSymbolE-Bold7
   <8-9>  MnSymbolE-Bold8
   <9-10> MnSymbolE-Bold9
  <10-12> MnSymbolE-Bold10
  <12->   MnSymbolE-Bold12
}{}

\let\llangle\@undefined
\let\rrangle\@undefined
\DeclareMathDelimiter{\llangle}{\mathopen}%
                     {MnLargeSymbols}{'164}{MnLargeSymbols}{'164}
\DeclareMathDelimiter{\rrangle}{\mathclose}%
                     {MnLargeSymbols}{'171}{MnLargeSymbols}{'171}
                     
\DeclareMathOperator*{\esssup}{ess\,sup}
\DeclareMathOperator*{\essinf}{ess\,inf}
\makeatother

\newenvironment{sciabstract}{}

\cftsetindents{section}{0em}{1.7em}

\sectionfont{\large}

\titlespacing*{\section}
{0pt}{5pt}{5pt}

\begin{document}

\maketitle

\begin{sciabstract}
  \textbf{Abstract.}
	The main goal of this paper is to generalize a part of the relationship between mean curvature and Harder-Narasimhan filtrations of holomorphic vector bundles to arbitrary polarized fibrations. 
	More precisely, for a polarized family of complex projective manifolds, we establish lower bounds on a fibered version of Yang-Mills functionals in terms of the Harder-Narasimhan slopes of direct image sheaves associated with high tensor powers of the polarization.
	We discuss the optimality of these lower bounds and, as an application, provide an analytic characterisation of a fibered version of generic nefness.
	As another application, we refine the existent obstructions for finding metrics with constant horizontal mean curvature. 
	The study of the semiclassical limit of Hermitian Yang-Mills functionals lies at the heart of our approach.
\end{sciabstract}

\pagestyle{fancy}
\lhead{}
\chead{Lower bounds on fibered Yang-Mills functionals}
\rhead{\thepage}
\cfoot{}

%\fancypagestyle{mypagestyle}{%
%  \fancyhf{}% Clear header/footer
%  \fancyhead[OC]{An Author}% Author on Odd page, Centred
%  \fancyhead[EC]{A titlesdfdsfdsfds}% Title on Even page, Centred
%  \fancyfoot[C]{\thepage}%
%  \renewcommand{\headrulewidth}{.4pt}% Header rule of .4pt
%}
%\pagestyle{mypagestyle}

\newcommand{\Addresses}{{% additional braces for segregating \footnotesize
  \bigskip
  \footnotesize
  \noindent \textsc{Siarhei Finski, CNRS-CMLS, École Polytechnique F-91128 Palaiseau Cedex, France.}\par\nopagebreak
  \noindent  \textit{E-mails }: \texttt{finski.siarhei@gmail.com} $\quad$ or  $\quad$  \texttt{siarhei.finski@polytechnique.edu}.
}} 

\vspace*{0.25cm}

\par\noindent\rule{1.25em}{0.4pt} \textbf{Table of contents} \hrulefill

\vspace*{-1.5cm}

\tableofcontents

\vspace*{-0.2cm}

\noindent \hrulefill

%\vspace*{-0.5cm}

\section{Introduction}\label{sect_intro}
	Consider a holomorphic submersion $\pi : X \to B$ between compact complex manifolds $X$ and $B$ of dimensions $n + m$ and $m$ respectively, $n, m \in \nat$.
	Let $L$ be a holomorphic line bundle over $X$, which is relatively ample with respect to $\pi$.
	We fix a \textit{Gauduchon Hermitian form} $\omega_B$ on $B$, i.e. a positive $(1, 1)$-form, such that $\partial \dbar \omega_B^{m - 1} = 0$, see \cite{Gauduchon}.
	The main goal of this paper is to study the relationship between the so-called horizontal mean curvature of the fibration, which is a certain differential-geometric invariant of the family defined using $\omega_B$, and Harder-Narasimhan $\omega_B$-slopes of direct images $E_k := R^0 \pi_* L^k$, which are algebraic invariants.
	\par 
	More precisely, consider a Hermitian metric $h^L$ on $L$, which is positive along the fibers of $\pi$.
	We denote by $\omega(h^L) := \frac{\imun}{2 \pi} R^L$ the first Chern form of $(L, h^L)$, where $R^L$ is the curvature of the Chern connection.
	When $h^L$ is clear from the context, we omit it from the above notation.
	\par
	As $\omega$ is positive along the fibers, it provides a (smooth) decomposition of the tangent space $TX$ of $X$ into the vertical component $T^V X$, corresponding to the tangent space of the fibers, and the horizontal component $T^H X$, corresponding to the orthogonal complement of $T^V X$ with respect to $\omega$.
	The form $\omega$ then decomposes as $\omega = \omega_V + \omega_H$, $\omega_V \in \ccal^{\infty}(X, \wedge^{1, 1} T^{V*} X)$, $\omega_H \in \ccal^{\infty}(X, \wedge^{1, 1} T^{H*} X)$. 
	Upon the natural identification of $T^H X$ with $\pi^* TB$, we may view $\omega_H$ as an element from $\ccal^{\infty}(X, \wedge^{1, 1} \pi^* T^* B)$.
	The triple $(\pi, \omega, T^H X)$ then defines a \textit{Kähler fibration} in the sense of \cite[Definition 1.4]{BGS2}.
	We define the \textit{horizontal mean curvature}, $\wedge_{\omega_B} \omega_H(h^L) \in \ccal^{\infty}(X)$, as
	\begin{equation}\label{eq_ma_fhe}
		\wedge_{\omega_B} \omega_H(h^L) := \frac{\omega_H(h^L) \wedge \omega_B^{m - 1}}{\omega_B^m}.
	\end{equation}
	\par 
	We say that $h^L$ is \textit{fibered Einstein} if $\wedge_{\omega_B} \omega_H(h^L)$ is a constant.
	By decomposition into horizontal and vertical components, it is easy to see that this condition is equivalent to
	\begin{equation}\label{eq_FE}
		\omega(h^L)^{n + 1} \wedge \pi^* \omega_B^{m - 1} = c \cdot \omega(h^L)^n \wedge \pi^* \omega_B^m,
	\end{equation}
	where $c$ is a constant.
	By integrating (\ref{eq_FE}), we see that the constant $c$ is independent of $h^L$, since $\omega_B$ is Gauduchon, see Section \ref{sect_num_obs} for details.
	Remark the similarity of (\ref{eq_FE}) with the J-equation if $m = 1$, and with optimal symplectic connection equation of Dervan-Sektnan, see \cite[Proposition 2.7]{DervSektUniqueOSC}, if the fibers are Fano.
	For families of manifolds given by projectivizations of vector bundles, the condition (\ref{eq_FE}) was introduced by Kobayashi \cite{KobFinEins}, who called such metrics \textit{Finsler-Einstein} metrics; then Feng-Liu-Wan \cite{FengLiuWang} generalized it for general Kähler submersions, and called such metrics \textit{geodesic Einstein} metrics.
	If instead of a closed manifold $B$, one considers manifolds with boundary, the equation (\ref{eq_FE}) was studied extensively in the past: if $B$ is a 1-dimensional annuli and $c = 0$, this is a geodesic equation in Mabuchi space \cite{Mabuchi}, see Semmes \cite{Semmes} or Donaldson \cite{DonaldSymSp}, if $B$ is a bounded smooth strongly pseudoconvex domain in $\comp^n$, $c = 0$ and $\omega_B$ is the standard Kähler form, (\ref{eq_FE}) was called Wess-Zumino-Witten equation in \cite{DonaldSymSp} due to its connection with \cite{WittenWZW}.
	\par
	Remark that in the important case when $X := \mathbb{P}(E^*)$ for some holomorphic vector bundle $E$ over $B$, $L := \mathscr{O}(1)$, and $h^L$ is induced by a Hermitian metric $h^E$ on $E$, $(L, h^L)$ is fibered Einstein if and only if $(E, h^E)$ is Hermite-Einstein, cf. Remark \ref{eq_he_fe_rel} for details.
	\par 
	The first main observation of this paper is that the correspondence between fibered Einstein and Hermite Einstein equations is much tighter. 
	Indeed, as we shall see, relying on the work of Ma-Zhang \cite{MaZhangSuperconnBKPubl}, cf. Theorem \ref{thm_mazh}, the fibered Einstein equation for $L$ is \textit{mutatis mutandis} the semi-classical limit (i.e. $k \to \infty$) of the Hermite-Einstein equation for $E_k := R^0 \pi_* L^k$. 
	Let us now explain the first manifestation of this correspondence.
	\par 
	\begin{sloppypar}
	Recall that a \textit{slope} (or $\omega_B$-slope) of a coherent sheaf $\mathscr{E}$ over $B$ is defined as $\mu(\mathscr{E}) := \deg(\mathscr{E}) / \rk{\mathscr{E}}$, where the degree, $\deg(\mathscr{E})$, is given by $\deg(\mathscr{E}) := \int_B [c_1(\det \mathscr{E})] \cdot [\omega_B^{m - 1}]$; here and after the intersection product is for Bott-Chern and Aeppli cohomology classes, $\omega_B^{m - 1}$ represents an Aeppli cohomology class since $\omega_B$ is Gauduchon, see Section \ref{sect_num_obs} for details, and $\det \mathscr{E}$ is the Knudsen-Mumford determinant of $\mathscr{E}$, see \cite{Knudsen1976}.
	A torsion-free coherent sheaf $\mathscr{E}$ is called \textit{semistable} or $\omega_B$-\textit{semistable} if for every coherent subsheaf $\mathcal{F}$ of $\mathscr{E}$, verifying $\rk{\mathcal{F}} > 0$, we have $\mu(\mathcal{F}) \leq \mu(\mathscr{E})$.
	When $\dim B = 1$, these notions clearly do not depend on $\omega_B$.
	\begin{comment}
	The next result shows that the fibered Einstein condition is tightly related with slope semistability of $E_k$.
	To readers familiar with Kobayashi-Hitchin correspondence, see Narasimhan-Seshadri \cite{NarSesh}, Donaldson \cite{DonaldASD} and Uhlenbeck-Yau \cite{UhlYau}, this will sound very natural.  
	Indeed, as Kobayashi-Hitchin correspondence relates the existence of approximate Hermite-Einstein metrics with slope semistability, by the mentioned correspondence between the horizontal mean curvature of $L$ and the mean curvature of $E_k$, $k \in \nat$, the existence of approximate fibered Einstein metric should be related with asymptotic semistability of $E_k$.
	\end{comment}
	\par 
	\begin{thm}\label{thm_singleton}
		Assume that $L$ admits an approximate fibered Einstein metric, i.e. there is $c \in \real$, such that for any $\epsilon > 0$, there is a relatively positive metric $h^L_{\epsilon}$ on $L$, verifying the following bound
		\begin{equation}
			\big| \wedge_{\omega_B} \omega_H(h^L_{\epsilon}) - c \big| < \epsilon.
		\end{equation}
		Then the vector bundles $E_k := R^0 \pi_* L^k$ are \textit{asymptotically semistable}, i.e. for any quotient sheaves $\mathcal{Q}_k$ of $E_k$, $\rk{\mathcal{Q}_k} > 0$, and any $\epsilon > 0$ for $k$ big enough, we have $\mu(\mathcal{Q}_k) \geq \mu(E_k) - \epsilon k$.
	\end{thm}
	\end{sloppypar}
	\begin{rem}
		a) It is likely that there is an even closer relationship between fibered Einstein metrics on $L$ and Hermite Einstein metrics on $E_k$, paralleling the known correspondence between constant scalar curvature and balanced metrics, see Donaldson \cite{DonaldScalBalan}.
		\par 
		b) We conjecture that the converse of Theorem \ref{thm_singleton} also holds. In fact, this will follow as a special case of a more general conjecture, which we discuss after Theorem \ref{thm_obs}.
	\end{rem}
	When Theorem \ref{thm_singleton} is applied to $\pi : \mathbb{P}(E^*) \to B$ for some vector bundle $E$ over $B$, and $L := \mathscr{O}(1)$, due to a precise relation between the slopes of $\sym^k E = R^0 \pi_* L^k$, $k \in \nat$, and $E$, see \cite[\S 3.2]{ChenVolume}, we recover the well-known fact, cf. \cite[Theorem 6.10.13]{KobaVB}, that if $E$ admits approximate Hermite-Einstein metrics, then $E$ is semistable.
	\par 
	The asymptotic semistability condition from Theorem \ref{thm_singleton} seems rather difficult to verify at first sight. 
	We will now discuss some numerical obstructions for it. 
	More precisely, for an irreducible complex analytic subspace $Y \subset X$ of dimension $k + m$, $k \geq 0$, such that the restriction of $\pi$ to $Y$, $\pi|_Y : Y \to B$, is a surjection, we define the $\omega_B$-slope, $\mu(Y)$, as
	\begin{equation}\label{eq_defn_slope_subm}
		\mu(Y) = \frac{1}{k + 1} \cdot \frac{\int_Y [c_1(L)^{k + 1}] \cdot [\omega_B^{m - 1}]}{\int_Y [c_1(L)^k] \cdot [\omega_B^m]}.
	\end{equation}
	By Serre vanishing theorem, for $k \in \nat^*$ big enough, the sheaf $\mathcal{Q}_k := R^0 \pi|_{Y, *} L|_Y^k$ can be realized as a quotient of $E_k$ through the restriction map, see proof of Proposition \ref{prop_mu_min_bnd} for details.
	By the asymptotic version of Riemann-Roch-Grothendieck theorem, see Theorem \ref{thm_rrg_asym}, which we establish in our singular setting, we can calculate the asymptotics of the slopes of $\mathcal{Q}_k$ and $E_k$, as $k \to \infty$.
	By comparing the asymptotics of these slopes, we obtain in Section \ref{sect_num_obs} the following result.
	\begin{thm}\label{thm_nonlinstab}
		If the vector bundles $E_k$ are asymptotically semistable, then $X$ is \textit{numerically semistable}, i.e. for any $Y$ as above, we have
		$
			\mu(Y) \geq \mu(X)
		$.
		Moreover, if $\dim B = 1$, then asymptotic semistability of $E_k$ is equivalent to numerical semistability of $X$.
	\end{thm}
	\begin{rem}
		A combination of Theorems \ref{thm_singleton} and \ref{thm_nonlinstab} shows that existence of  approximate fibered Einstein metrics on $L$ implies $\mu(Y) \geq \mu(X)$ for $Y$ above.
		Feng-Liu-Wan \cite[Theorem 2.2]{FengLiuWang}, cf. also Wan-Wang \cite{WanWangRem}, established this by different means under an assumption, requiring among others that the projection of the singular locus of $Y$ to $B$ has codimension at least 2.
		\begin{comment}
		b)
		An attentive reader will notice two subtleties in the calculation of the asymptotics of the slopes of $\mathcal{Q}_k$ and $E_k$, as $k \to \infty$. 
		First, as $\omega_B$ is only Gauduchon, we need to use Riemann-Roch-Grothendieck type theorems in Bott-Chern cohomology. 
		Second, the target manifold in the Riemann-Roch-Grothendieck theorem for $\mathcal{Q}_k$ is not smooth in general, and so these Riemann-Roch-Grothendieck type theorems cannot be applied directly, see Section ... for details.

		These technical difficulties are resolved by using the recent result of Bismut-Shen-Wei \cite{BismutShuWei} and the result of Hironaka \cite[Corollary 2.26]{HironakaI} about the vanishing of higher direct images of birational maps between smooth manifolds.
		\end{comment}
	\end{rem}
	\par 
	As we explain later, Theorem \ref{thm_singleton} is a direct consequence of a more refined result concerning lower bounds on fibered Yang-Mills functionals.
	More precisely, for a relatively Kähler $(1, 1)$-form $\omega$ on $X$ and any $c \in \real$, $p \in [1, +\infty[$, we define the fibered Yang-Mills functional as
	\begin{equation}
		\begin{aligned}
		&
		{\rm{FYM}}_{p, c}(\pi, \omega)
		:=
		\int_X \big| \wedge_{\omega_B} \omega_H(x) - c  \big|^p \omega^n \wedge \pi^* \omega_B^m (x),
		\\
		&
		{\rm{FYM}}_{+ \infty, c}(\pi, \omega)
		:=
		\sup_{x \in X} \big| \wedge_{\omega_B} \omega_H(x) - c  \big|.
		\end{aligned}
	\end{equation}
	We also denote ${\rm{FYM}}_{p, c}(\pi, h^L) := {\rm{FYM}}_{p, c}(\pi, c_1(L, h^L))$ for a relatively positive metric $h^L$ on $L$.
	We will now show that asymptotic Harder-Narasimhan slopes of $E_k$, as $k \to \infty$, yield lower bounds for these functionals.
	To readers familiar with Hermitian Yang-Mills theory, see Atiyah-Bott \cite[Proposition 8.20]{AtiyahBott}, Donaldson \cite[Proposition 5]{DonaldASD} and Daskalopoulos-Wentworth \cite[\S 2.3, 2.4]{DaskWent}, this will sound very natural.
	Indeed, again from the work of Ma-Zhang \cite{MaZhangSuperconnBKPubl}, one can view the horizontal mean curvature of $L$ as the semi-classical limit (i.e. $k \to \infty$) of the mean curvature of $E_k$. 
	From this, we establish the lower bounds on the fibered Yang-Mills functionals through the limits of the lower bounds on the Hermitian Yang-Mills functionals of $E_k$. 
	\par 
	To explain this in details, recall first that any torsion-free coherent sheaf $\mathscr{E}$ on $(B, [\omega_B])$ admits a unique filtration by subsheaves $\mathscr{E}_i$, $i = 1, \ldots, s$, also called \textit{Harder-Narasimhan filtration}:
	\begin{equation}\label{eq_HN_filt}
		\mathscr{E} = \mathscr{E}_s \supset \mathscr{E}_{s - 1} \supset \cdots \supset \mathscr{E}_1 \supset \mathscr{E}_0 = \{0\},
	\end{equation}
	such that for any $1 \leq i \leq s - 1$, the quotient sheaf $\mathscr{E}_i / \mathscr{E}_{i - 1}$ is the maximal semistable (torsion-free coherent) subsheaf of $\mathscr{E} / \mathscr{E}_{i - 1}$, i.e. for any subsheaf of $\mathcal{F}$ of a (torsion-free coherent) sheaf $\mathscr{E} / \mathscr{E}_{i - 1}$, we have $\mu(\mathcal{F}) \leq \mu(\mathscr{E}_i / \mathscr{E}_{i - 1})$ and $\rk{\mathcal{F}} \leq \rk{\mathscr{E}_i / \mathscr{E}_{i - 1}}$ if $\mu(\mathcal{F}) = \mu(\mathscr{E}_i / \mathscr{E}_{i - 1})$.
	For the proof of this result in the setting of Gauduchon form $\omega_B$, see either Bruasse \cite{BruasseHN} or Greb-Kebekus-Peternell \cite[Corollary 2.27]{GrebKebPetMov}.
	We define the \textit{Harder-Narasimhan slopes}, $\mu_1, \ldots, \mu_{\rk{\mathscr{E}}}$ of $\mathscr{E}$, so that $\mu(\mathscr{E}_i / \mathscr{E}_{i - 1})$ appears among $\mu_1, \ldots, \mu_{\rk{\mathscr{E}}}$ exactly $\rk{\mathscr{E}_i / \mathscr{E}_{i - 1}}$ times, and the sequence $\mu_1, \ldots, \mu_{\rk{\mathscr{E}}}$ is non-increasing.
	We let $\mu_{\min} := \mu_{\rk{\mathscr{E}}}$, $\mu_{\max} := \mu_{1}$. 
	\par 
	We denote $N_k := \rk{E_k}$, and let $\mu_1^k, \ldots, \mu_{N_k}^k$ be the Harder-Narasimhan slopes of $E_k$, and $\mu_{\min}^k, \mu_{\max}^k$ be the minimal and the maximal slopes.
	Define the probability measure $\eta_k^{HN}$ on $\real$ as
	\begin{equation}\label{eq_eta_defn}
		\eta_k^{HN} := \frac{1}{N_k} \sum_{i = 1}^{N_k} \delta \Big[ \frac{\mu_i^k}{k} \Big],
	\end{equation}
	where $\delta[x]$ is the Dirac mass at $x \in \real$. Our lower bounds for the fibered Yang-Mills functionals will build upon the following result. 
	\begin{thm}\label{thm_conv_meas}
		The sequence of measures $\eta_k^{HN}$ converges weakly, as $k \to \infty$, to a probability measure $\eta^{HN}$ on $\real$, and the limits below exist and relate with $\eta^{HN}$ as follows 
		\begin{equation}\label{eq_conv_meas}
			\eta_{\min}^{HN} := \lim_{k \to \infty} \frac{\mu_{\min}^k}{k} \leq \essinf \eta^{HN},
			\qquad
			\eta_{\max}^{HN} := \lim_{k \to \infty} \frac{\mu_{\max}^k}{k} = \esssup \eta^{HN}.
		\end{equation}
	\end{thm}
	\begin{rem}
		The proof of Theorem \ref{thm_conv_meas} follows the arguments from Chen \cite{ChenHNolyg} and \cite[Theorem 1.1]{FinHNI}, establishing Theorem \ref{thm_conv_meas} in the projective setting for flat maps $\pi: X \to B$, for $\dim B = 1$ and $\dim B \geq 1$ respectively. 
		The only difference is that due to the lack of algebraicity, the proofs from \cite{ChenHNolyg} and \cite{FinHNI} of the linear upper bound on $\mu_{\max}^k$ in $k \in \nat^*$, crucial for Theorem \ref{thm_conv_meas}, do not work. 
		Here this bound is obtained by a differential-geometric argument, see Proposition \ref{prop_bound_slope}.
	\end{rem}
	We are finally ready to state our lower bounds for the fibered Yang-Mills functionals.
	\begin{thm}\label{thm_obs}
		For any relatively positive metric $h^L$ on $L$, we have
		\begin{equation}\label{eq_ess_sup_inf1}
			\inf_{x \in X} \wedge_{\omega_B} \omega_H(x) \leq \eta_{\min}^{HN}, 
			\qquad 
			\sup_{x \in X} \wedge_{\omega_B} \omega_H(x) \geq \eta_{\max}^{HN},
		\end{equation}
		If, moreover, $\omega_B$ is Kähler, then for any $c \in \real$, $p \in [1, +\infty[$, we have
		\begin{equation}\label{eq_obs}
			{\rm{FYM}}_{p, c}(\pi, h^L)
			\geq
			\int_{\real} \big|x - c \big|^p d \eta^{HN}(x)
			\cdot
			\int_X [\omega^n] \cdot \pi^*[\omega_B^m].
		\end{equation}
	\end{thm}
	\begin{rem}\label{rem_thm_obs}
		a) As we shall establish in Proposition \ref{prop_asymp_semist}, $\eta_{\min}^{HN} = \eta_{\max}^{HN}$ if and only if $E_k$, $k \in \nat$, are asymptotically semistable.
		Hence, (\ref{eq_ess_sup_inf1}) refines Theorem \ref{thm_singleton}.
		\par 
		b) Remark that the left hand-side of (\ref{eq_obs}) depends on $h^L$, but the right-hand side doesn't.
		\par 
		c) When $X := \mathbb{P}(E^*)$ for some holomorphic vector bundle $E$ over $B$, $L := \mathscr{O}(1)$, and $h^L$ is induced by a Hermitian metric $h^E$ on $E$, the result can be deduced from the lower bounds on the Hermitian Yang-Mills functionals due to Atiyah-Bott \cite{AtiyahBott} and Daskalopoulos-Wentworth \cite{DaskWent}.
	\end{rem}
	\par 
	Remark that similar lower bounds in the context of constant scalar curvature metrics were obtained by Donaldson \cite{DonLowCalabi} for Calabi functional.
	Here, as in \cite{DonLowCalabi}, we expect the bounds from Theorem \ref{thm_obs} to be tight. In other words, it seems likely that the following conjecture holds.
	\vspace*{0.3cm}
	\par 
	\noindent \textbf{Conjecture.}
		\textit{In the notations of Theorem \ref{thm_obs}, for any $p \in [1, +\infty[$, $c \in \real$, we have}
		\begin{equation}\label{eq_obsid}
		\begin{aligned}
			&
			\inf_{h^L}
			{\rm{FYM}}_{p, c}(\pi, h^L)
			=
			\int_{\real} \big|x - c \big|^p d \eta^{HN}(x)
			\cdot
			\int_X [\omega^n] \cdot \pi^*[\omega_B^m],
			\\
			&
			\inf_{h^L}
			{\rm{FYM}}_{+ \infty, c}(\pi, h^L)
			=
			\max \Big\{ | \eta_{\min}^{HN} - c |, | \eta_{\max}^{HN} - c | \Big\},
		\end{aligned}
		\end{equation}
		\textit{where the infimum is taken among all relatively positive metrics $h^L$ on $L$.}
	\par 
	\vspace*{0.3cm}
	\begin{rem}
		In the recent paper \cite{FinHYM}, author established the above conjecture for $p = 1$.
	\end{rem}
	\par 
	By Remark \ref{rem_thm_obs}.a), if the above conjecture holds for $p = + \infty$, then the converse implication of Theorem \ref{thm_singleton} also follows, upon taking $c := \eta_{\min}^{HN} = \eta_{\max}^{HN}$.
	\par 
	For $\pi : \mathbb{P}(E^*) \to B$, $L = \mathscr{O}(1)$, where $E$ is some holomorphic vector bundle over a complex compact manifold $B$, one can show that Conjecture holds by the existence of the approximate critical hermitian structures on vector bundles, cf. Daskalopoulos-Wentworth \cite[Definition 3.9]{DaskWent}, and the calculation of the asymptotic slopes of $\sym^k E = R^0 \pi_* L^k$, due to Chen \cite[Theorem 1.2]{ChenVolume}.
	The following result is another partial justification of the Conjecture.
	\begin{thm}\label{thm_trshld}
		The following identity holds
		$
			\sup_{h^L} \inf_{x \in X} \wedge_{\omega_B} \omega_H(x) = \eta_{\min}^{HN},
		$
		where the supremum is taken among all relatively positive metrics $h^L$ on $L$.
	\end{thm}
	\begin{rem}
		When $X := \mathbb{P}(E^*)$ for some holomorphic vector bundle $E$ over $B$, $L := \mathscr{O}(1)$, and $h^L$ is induced by a Hermitian metric $h^E$ on $E$, follows from Li-Zhang-Zhang \cite[Theorem 1.5]{LiZhaingHN}.
	\end{rem}
	\par 
	\begin{sloppypar}
	We assume now that $B$ (and, hence, $X$) is projective.
	Recall that a line bundle $L$ on $X$ is called \textit{nef} if for any irreducible curve $C$ in $X$, $\int_{C} c_1(L) \geq 0$.
	It is well-known that this condition is equivalent to the existence of metrics with an arbitrarily small negative part of the curvature, as stated precisely in \cite[Proposition 4.2]{DemSingHermMetr}. 
	One of the central concerns in complex geometry is the study of variations of this result, which provides a “dictionary" between the algebraic and analytic definitions of positivity. 
	Let us now explain an application of Theorem \ref{thm_trshld} in this context.
	\par 
	We fix a very ample integral multipolarization $[\omega_{B, 1}], \ldots, [\omega_{B, m - 1}]$ on $B$, which is a collection of very ample classes from $H^{1, 1}(B, \comp) \cap H^2(B, \integ)$.
	We say that a $\mathbb{Q}$-line bundle $L$ on $X$ is \textit{$([\omega_{B, 1}], \ldots, [\omega_{B, m - 1}])$-generically fibered nef with respect to $\pi$} if there is $l_0 \in \nat^*$, such that for any regular curve $C = C(l) \subset B$, $l = (l_1, \ldots, l_{m - 1}) \in \nat^{*(m - 1)}$, $l_i \geq l_0$, $i = 1, \ldots, m - 1$, given by a complete intersection of \textit{generic} divisors from classes $l_1 [\omega_{B, 1}], \ldots, l_{m - 1} [\omega_{B, m - 1}]$, the restriction of $c_1(L)$ to $\pi^{-1}(C)$ is nef.
	When $\pi$ is the projectivization of a vector bundle, equivalent definition was given by Miyaoka \cite{MiyGenNef}, see also Peternell \cite{PetGenNef}.
	The general case was introduced in \cite{FinHNI}.
	We say $L$ is \textit{stably $([\omega_{B, 1}], \ldots, [\omega_{B, m - 1}])$-generically fibered nef with respect to $\pi$} if for some (or any) ample line bundle $L_0$ on $X$, for any $\epsilon > 0$, $\epsilon \in \mathbb{Q}$, the $\mathbb{Q}$-line bundle $L \otimes L_0^{\epsilon}$ is $([\omega_{B, 1}], \ldots, [\omega_{B, m - 1}])$-generically fibered nef with respect to $\pi$.
	Recall that $L$ is called \textit{relatively nef with respect to $\pi$} if its restriction to every fiber is nef.
	As we explain in Section \ref{sect_hor_curv}, from the previously obtained algebraic description of $\eta_{\min}^{HN}$ from \cite[Corollary 1.4]{FinHNI}, Theorem \ref{thm_trshld} can be used to prove the following result.
	\par
	\begin{thm}\label{thm_generic_nef}
		Consider a holomorphic submersion $\pi : X \to B$ between projective manifolds $B, X$. 
		A relatively nef line bundle $L$ on $X$ is stably $([\omega_{B, 1}], \ldots, [\omega_{B, m - 1}])$-generically fibered nef if and only if for any (or some) Kähler forms $\omega_{B, 1}, \ldots, \omega_{B, m - 1}$ on $B$ in $[\omega_{B, 1}], \ldots, [\omega_{B, m - 1}]$, and any (or some) Kähler form $\omega_{X}$ on $X$, for any $\epsilon > 0$, there is a Hermitian metric $h^L_{\epsilon}$ on $L$, such that
		\begin{equation}\label{eq_generic_nef}
			\omega(h^L_{\epsilon}) \wedge \pi^* \omega_{B, 1} \wedge \cdots \wedge \pi^* \omega_{B, m - 1}  
			\geq
			-\epsilon \cdot \omega_X^m,
		\end{equation}
		where by this we mean that the volume forms obtained by the restriction to every $m$-dimensional complex hyperplane of each side of (\ref{eq_generic_nef}) compares as required in (\ref{eq_generic_nef}), cf. \cite[(III.1.6)]{DemCompl}.
	\end{thm}
	\end{sloppypar}
	\begin{rem}
		Curiously, even though the forms $\omega_{B, 1}, \ldots, \omega_{B, m - 1}$ are Kähler, if these forms are different, then in the proof of Theorem \ref{thm_generic_nef}, we need to apply Theorem \ref{thm_trshld} for a non-Kähler Gauduchon form $\omega_B$, constructed from $\omega_{B, 1}, \ldots, \omega_{B, m - 1}$. 
		This was our main motivation to write this article in the current generality.
		However, Theorem \ref{thm_generic_nef} is new even if the forms are equal.
	\end{rem}
	\par 
	In conclusion, it seems for us that a proof of the Conjecture might rely either on the techniques of geometric flows or continuity method as in Donaldson \cite{DonaldASD} and Uhlenbeck-Yau \cite{UhlYau}.
	In this vein, the recent apriori bounds for Monge-Amp\`ere and Hessian equations established by Guo-Phong-Tong \cite{GuoPhongTong} and Guo-Phong-Tong-Wang \cite{GuoPhongTongWang} will likely play an important role.
	\par 
	Remark that for Hermitian Yang-Mills functional, the analogous conjecture holds due to results of Atiyah-Bott \cite{AtiyahBott}, Daskalopoulos-Wentworth \cite{DaskWent}, Sibley \cite{SibleyYangMills}, Jacob \cite{JacobYangMillsAB} and Li-Zhang-Zhang \cite{LiZhaingHN}, cf. Theorem \ref{thm_ym_tight} for a precise statement.
	We also mention the recent works of Xia \cite{XiaCalabi}, Hisamoto \cite{HisamCalabi} and Dervan-Sz{\'e}kelyhidi \cite{DerSzek}, cf. also Collins-Hisamoto-Takahashi \cite{CollinsHisamotoTaka}, proving versions of the above Conjecture in the context of constant scalar curvature metrics. 
	\par 
	In a different direction, when $B$ is a bounded smooth strongly pseudoconvex domain in $\comp^n$, $c = 0$ and $\omega_B$ is the standard Kähler form on $B \subset \comp^n$,
	Donaldson \cite{DonaldYMBoundVal} and Coifman-Semmes \cite{CoifmanSemmes} established that Dirichlet problem associated with the Hermite-Einstein equation has solutions for any vector bundle over $B$ (in particular for $E_k$, $k \in \nat$). 
	Wu in \cite{WuKRWZW} showed that Dirichlet problem associated with (\ref{eq_FE}) has always weak solutions, and these solutions can be obtained as the semiclassical limit of the solutions of the Hermite-Einstein equations on $E_k$, cf. also \cite{PhongSturm}, \cite{RubinZeld}, \cite{SongZeldToric} for earlier results in this direction.
	In other words, a phenomenon similar to Theorem \ref{thm_obs} is present: there is a relation between the Hermite-Einstein and the Wess-Zumino-Witten equations.
	The major difference between these developments and our paper is that in our boundaryless setting, neither Hermite-Einstein equations nor fibered Einstein equations have solutions in general, and the methods of \cite{DonaldYMBoundVal}, \cite{CoifmanSemmes}, \cite{WuKRWZW} do not apply.
	\par 
	This article will be organized as follows.
	In Section \ref{sect_curv_dir_im}, we will establish Theorems \ref{thm_conv_meas} and \ref{thm_obs}.
	We discuss how horizontal curvature behaves with respect to a restriction to a subfamily in Section \ref{sect_hor_curv}, and using this, we establish Theorems \ref{thm_trshld}, \ref{thm_generic_nef}.
	Finally, in Section \ref{sect_num_obs}, we establish a numerical obstruction for asymptotic semistability of direct images from Theorem \ref{thm_nonlinstab}.
	\par 
	\textbf{Acknowledgement}. 
	I had a privilege to discuss the results of this paper with Sébastien Boucksom, Paul Gauduchon, Duong H. Phong, Lars M. Sektnan, Jacob Sturm and Richard A. Wentworth, whom I thank for their interest.
	I acknowledge the support of CNRS and École Polytechnique as well as the partial support of ANR-23-CE40-0021-01 JCJC project QCM. 
	A part of this paper was written in the Fall of 2023 during a visit in Columbia University.
	I would like to thank the mathematical department of Columbia University, especially Duong H. Phong, for their hospitality and Alliance Program for their support.
	Finally, I would like to express my gratitude to the anonymous referees for their valuable comments.

	\section{Fibered Yang-Mills functionals through the semiclassical limit}\label{sect_curv_dir_im}
	The main goal of this section is to prove Theorems \ref{thm_conv_meas} and \ref{thm_obs}.
	The theory of Toeplitz operators and Hermitian Yang-Mills theory, which we recall below, will be particularly useful for that.
	\par 
	We begin by recalling some facts about Toeplitz operators. 
	Let $Y$ be a complex projective manifold of dimension $n$ with an ample line bundle $L$.
	We fix a positive Hermitian metric $h^L$ on $L$.
	We denote by $\omega$ its first Chern form, $c_1(L, h^L)$.
	For smooth sections $f, f'$ of $L^k$, $k \in \nat$, over $Y$, we define the $L^2$-scalar product using the pointwise scalar product $\scal{\cdot}{\cdot}_{h^L}$ induced by $h^L$ as follows
	\begin{equation}\label{eq_l2_prod}
		\scal{f}{f'}_{L^2(Y)} := \int_Y \scal{f(x)}{f'(x)}_{h^{L^{\otimes k}}} \cdot \omega^n(x).
	\end{equation}
	\par 
	Recall that the \textit{Bergman projector} $B_k$ is given by the orthogonal projection (with respect to the scalar product (\ref{eq_l2_prod})) from the space of $L^2$-sections of $L^k$ to $H^0(Y, L^k)$.
	For any bounded function $f$ on $Y$, we then define the Toeplitz operator, $T_k(f) : H^0(Y, L^k) \to H^0(Y, L^k)$, as follows 
	\begin{equation}
		T_k(f)(s) = B_k(f \cdot s), \qquad s \in H^0(Y, L^k).
	\end{equation}
	\begin{prop}\label{prop_toepl}
		For any bounded function $f : Y \to \real$, the following inequalities hold
		\begin{equation}\label{eq_toepl_inf_sup}
			\inf f \cdot {\rm{Id}}_{H^0(Y, L^k)} \leq T_k(f) \leq \sup f \cdot {\rm{Id}}_{H^0(Y, L^k)},
		\end{equation}
		where by $A \leq B$ we mean that the difference $B - A$ is positive definite.
		Moreover, if $f$ is smooth, then for any continuous function $\phi : \real \to \real$, we have
		\begin{equation}\label{eq_toepl_tr}
			\begin{aligned}
			&\lim_{k \to \infty} \frac{\tr{\phi(T_k(f))}}{\dim H^0(Y, L^k)} = \frac{\int_{x \in Y} \phi(f(x)) \omega^n(x)}{\int_Y [\omega^n]},
			\\ 
			&\lim_{k \to \infty} \big\| \phi(T_k(f)) \big\| = \max \big\{ |\sup \phi(f)|, |\inf \phi(f)| \big\},
			\end{aligned}
		\end{equation}
		where $\| \cdot \|$ is the operators norm.
	\end{prop}
	\begin{proof}
		The statement (\ref{eq_toepl_inf_sup}) follows from the trivial fact that if $f$ is a positive function, then the operator $T_k(f)$ is positive-definite.
		The statement (\ref{eq_toepl_tr}) is a restatement of the weak convergence of spectral measures of Toeplitz operators due to Boutet de Monvel-Guillemin \cite[Theorem 13.13]{BoutGuillSpecToepl}.
		For an alternative proof through Bergman kernel expansion, see Ma-Marinescu \cite[Theorem 7.4.1]{MaHol}, Barron-Ma-Marinescu-Pinsonnault \cite[Theorem 3.8]{BarrMa} or \cite[Appendix A]{FinMaVol}.
		See also Ma-Marinescu \cite{MaMarBTKah}, \cite{MaMarSecondTerm} for generalizations and more refined results. 
	\end{proof}
	\par 
	Now, the reason why Toeplitz operators are relevant to this paper is because they appear as the principal term in the asymptotic expansion of the curvature of $L^2$-metrics on direct images of a polarized fibrations.
	More precisely, consider a proper holomorphic submersion $\pi : X \to B$ between complex manifolds $X$ and $B$ of dimensions $n + m$ and $m$ respectively, $n, m \in \nat$.
	Let $L$ be a holomorphic line bundle over $X$, which is relatively ample with respect to $\pi$.
	Endow $L$ with a relatively positive Hermitian metric $h^L$.
	Let $k \in \nat$ be big enough so that 
	\begin{equation}
		E_k := R^0 \pi_* L^k
	\end{equation}
	is locally free.
	The $L^2$-product (\ref{eq_l2_prod}) then defines a smooth Hermitian metric $h^{E_k}$ on $E_k$.
	We denote by $R^{E_k} \in \ccal^{\infty}(B, \wedge^2 T^*B \otimes \enmr{E_k})$ the curvature of its Chern connection.
	\begin{thm}[{ Ma-Zhang \cite[Theorem 0.4]{MaZhangSuperconnBKPubl} }]\label{thm_mazh}
		There are $C > 0$, $k_0 \in \nat$, such that for any $k \geq k_0$,
		\begin{equation}
			\Big\|
				\frac{\imun}{2 \pi}
				R^{E_k}
				-
				k \cdot T_k(\omega_H)
			\Big\|
			\leq
			C,
		\end{equation}
		where $\| \cdot \|$ is the operator norm, and we naturally extended the definition of Toeplitz operators from functions to bounded sections of $\pi^* \wedge^2 T^*B$ as follows: for a decomposition $\omega_H = \sum f_{i j} dz_i d\overline{z}_j$, where $z_1, \ldots, z_n$ are local coordinates on $B$, we let $T_k(\omega_H) := \sum T_k(f_{i j}) dz_i d\overline{z}_j$.
	\end{thm}
	As we shall explain below, Theorem \ref{thm_mazh} is the crucial ingredient connecting fibered Yang-Mills functionals with Hermitian Yang-Mills functionals.
	But before this, let us mention another application of Theorem \ref{thm_mazh} to the study of Harder-Narasimhan slopes of direct images.
	\par 
	We fix now a \textit{Gauduchon Hermitian form} $\omega_B$ on $B$.
	As before Theorem \ref{thm_conv_meas}, we denote by $\mu_{\max}^k$ the maximal Harder-Narasimhan $\omega_B$-slope of $E_k$.
	\begin{prop}\label{prop_bound_slope}
		There is $C > 0$, such that $\mu_{\max}^k \leq C k$ for any $k \in \nat^*$.
	\end{prop}	 
	\begin{proof}
		For $p = 1, \ldots, \rk{E_k}$, we denote by $R^{\wedge^p E_k}$ the curvature of the Chern connection on $\wedge^p E_k$, induced by the metric $h^{\wedge^p E_k}$ induced by $h^{E_k}$. 
		By Theorem \ref{thm_mazh} and the very definition of $\wedge_{\omega_B}$ from (\ref{eq_ma_fhe}), we conclude that there is $C > 0$, such that for any $k \in \nat^*$, we have
		\begin{equation}\label{eq_bnd_wedge}
			\frac{\imun}{2 \pi} \wedge_{\omega_B} R^{\wedge^p E_k} \leq C p k \cdot {\rm{Id}}_{\wedge^p E_k}.
		\end{equation}
		Let $F$ be a line subbundle of $\wedge^p E_k$.
		We denote by $h^F$ the Hermitian metric on $F$, induced by the metric $h^{E_k}$.
		By (\ref{eq_bnd_wedge}) and the well-known principle that curvature decreases in holomorphic sub-bundles, cf. \cite[(V.14.6)]{DemCompl}, we deduce
		\begin{equation}\label{eq_bound_slope1}
			\wedge_{\omega_B} c_1(F, h^F)
			\leq
			C p k.
		\end{equation}
		However, it is classical, cf. \cite[Proofs of Lemma 5.7.16 and Theorem 5.8.3]{KobaVB}, that we have
		\begin{equation}\label{eq_bound_slope2}
			\mu_{\max}^k \leq \max_{p = 1, \ldots, \rk{E_k}} \sup_{F \subset \wedge^p E_k} \frac{1}{p} \int_B c_1(F, h^F) \wedge \omega_B^{m - 1},
		\end{equation}
		where the second supremum is taken over line subbundles $F$.
		We conclude by (\ref{eq_bound_slope1}) and (\ref{eq_bound_slope2}).
	\end{proof}
	
	\begin{proof}[Proof of Theorem \ref{thm_conv_meas}]
		Taking into account the linear bound from Proposition \ref{prop_bound_slope}, the proof of Theorem \ref{thm_conv_meas} is the same as in Chen \cite{ChenHNolyg} and \cite[Theorem 1.1]{FinHNI}.
		Let us briefly recall the main steps for completeness.
		We introduce the (non-increasing) filtrations $\mathcal{F}_k(\lambda)$, $\lambda \in \real$, of $E_k$ by coherent (torsion-free) subsheaves (defined over $B$), so that $\mathcal{F}_k(\lambda)$ is the maximal subsheaf of $E_k$ such that all of its Harder-Narasimhan slopes are bigger than $\lambda$.
		The filtration $\mathcal{F}_k$ is just a “renaming" of the Harder-Narasimhan filtration of $E_k$.
		Now, for any $b \in B$, we denote by $\mathcal{F}_b$ the filtration induced by $\mathcal{F}_k(\lambda)$ on $R(X_b, L_b) = \oplus_{k = 0}^{\infty} H^0(X_b, L_b^{k})$ of the fiber $X_b = \pi^{-1}(b)$, $L_b = L|_{X_b}$, $b \in B$.
		It was established in \cite{ChenHNolyg} for $\dim B = 1$ and in \cite[Proposition 2.5]{FinHNI} for any projective $B$, that for generic $b \in B$, the above filtration is submultiplicative, i.e. for any $t, s \in \real$, $k, l \in \nat$, we have
		\begin{equation}
			\mathcal{F}_b^t H^0(X_b, L_b^{k}) \cdot \mathcal{F}_b^s H^0(X_b, L_b^{l}) \subset \mathcal{F}_b^{t + s} H^0(X_b, L_b^{k + l}).
		\end{equation}
		Remark, however, that the projectivity assumption was never used in \cite[Proposition 2.5]{FinHNI}, and so submultiplicativity holds for general complex manifolds $B$.
		Theorem \ref{thm_conv_meas} is then a formal consequence of the submultiplicativity and Proposition \ref{prop_bound_slope}, saying that the above filtration is bounded in the terminology of \cite{BouckChen}.
		For the (different) proofs of this last result, see \cite[Théorème 3.4.3]{ChenHNolyg}, \cite[Theorem A]{BouckChen} and \cite[Theorem 1.9]{FinNarSim}.
	\end{proof}
	
	Let us now recall some crucial facts from Hermitian Yang-Mills theory, following the pioneering work of Atiyah-Bott \cite{AtiyahBott} and later developments by Donaldson \cite{DonaldASD}, Daskalopoulos-Wentworth \cite{DaskWent} and others.
	We fix a compact complex manifold $B$ of dimension $m$ with a Gauduchon Hermitian form $\omega_B$ on $B$.
	Let $E$ be a holomorphic vector bundle of rank $r$ over $B$.
	For a Hermitian metric $h^E$ on $E$, we denote by $R^E$ its curvature.
	For any $p \in [1, +\infty[$, $c \in \real$, we define the \textit{Hermitian Yang-Mills functional} as
	\begin{equation}
		\begin{aligned}
		&
		{\rm{HYM}}_{p, c}(E, h^E) 
		:=
		\int_B {\rm{Tr}} \Big[ \Big| \frac{\imun}{2 \pi} \wedge_{\omega_B} R_x^E - c \cdot {\rm{Id}}_E \Big|^p \Big] \omega_B^m(x),
		\\
		&
		{\rm{HYM}}_{+\infty, c}(E, h^E) 
		:=
		\sup_{x \in X} \Big\| \frac{\imun}{2 \pi} \wedge_{\omega_B} R_x^E - c \cdot {\rm{Id}}_E \Big\|,
		\end{aligned}
	\end{equation}
	where $\| \cdot \|$ means the operator norm, and $|A| := \sqrt{A A^*}$ for $A \in \enmr{V}$ on a Hermitian vector space $(V, H)$.
	\par
	As in (\ref{eq_eta_defn}), we denote the Harder-Narasimhan $\omega_B$-slopes of $E$ by $\mu_1, \ldots, \mu_r$. 
	Let $\mu_{\min} := \mu_r$, $\mu_{\max} := \mu_{1}$, be the minimal and the maximal slopes. We define the probability measure
	\begin{equation}
		\mu_E := \frac{1}{r} \sum_{i = 1}^{r} \delta[\mu_i].
	\end{equation}
	The following result lies at the heart of this paper.  
	\begin{thm}\label{thm_ym}
		For any $c \in \real$ and a Hermitian metric $h^E$ on $E$, we have
		\begin{equation}
			{\rm{HYM}}_{+\infty, c}(E, h^E) 
			\geq
			\max \Big\{ | \mu_{\min} - c |, | \mu_{\max} - c | \Big\}.
		\end{equation}
		If, moreover, $\omega_B$ is Kähler, then for any $p \in [1, +\infty[$, we have
		\begin{equation}\label{eq_obs_ymmm}
			{\rm{HYM}}_{p, c}(E, h^E) 
			\geq
			r
			\cdot
			\int_{\real} |x - c|^p  d \mu_E(x)
		\end{equation}
	\end{thm}
	For the proof of Theorem \ref{thm_ym} for $p \in [1, +\infty[$, see Atiyah-Bott \cite[Proposition 8.20]{AtiyahBott} if $\dim B = 1$, Daskalopoulos-Wentworth \cite[Lemma 2.17, Corollary 2.22, Proposition 2.25]{DaskWent} if $B$ is Kähler of any dimension (even though the article \cite{DaskWent} is written for surfaces, cf. Sibley \cite[\S 3.1]{SibleyYangMills}).
	For the proof of the first part, consult Li-Zhang-Zhang \cite[Theorem 1.5]{LiZhaingHN}.
	It is a remarkable that the bounds from Theorem \ref{thm_ym} are actually tight.
	Although we will not use this result in what follows, we state it for the reader's convenience, as it clarifies our motivation for the Conjecture.
	\begin{thm}\label{thm_ym_tight}
		In the notations Theorem \ref{thm_ym}, assume that $\omega_B$ is Kähler. 
		Then for any $c \in \real$, $p \in [1, +\infty[$, we have
		\begin{equation}\label{eq_ym_tight1}
			\inf_{h^E} {\rm{HYM}}_{p, c}(E, h^E) 
			=
			r
			\cdot
			\int_{\real} |x - c|^p  d \mu_E(x)
			\cdot
			\int_B [\omega_B^m].
		\end{equation}
		where the infimum is taken over all Hermitian metrics $h^E$ on $E$.
	\end{thm}
	For the proof of Theorem \ref{thm_ym_tight} for $p \in [1, +\infty[$, see \cite[Proposition 8.20]{AtiyahBott} if $\dim B = 1$.
	For higher dimensions, this is a direct consequence of the existence of $L^p$-approximate critical hermitian structure on $E$, see \cite[Definition 3.9]{DaskWent} for the definition, and \cite[Theorem 3.11]{DaskWent} and Sibley \cite[Theorem 1.3]{SibleyYangMills} for the proofs if $B$ is Kähler of dimension $2$ and any dimension respectively. 
	See also Jacob \cite[Theorems 2, 3]{JacobYangMillsAB}.
	\begin{proof}[Proof of Theorem \ref{thm_obs}]
		Preserving the notation introduced in (\ref{eq_eta_defn}), we define the probability measure, $\eta_{k, 0}^{HN}$, $k \in \nat$, on $\real$ as
		\begin{equation}\label{eq_eta_defn22}
			\eta_{k, 0}^{HN} := \frac{1}{N_k} \sum_{i = 1}^{N_k} \delta \big[ \mu_i^k \big],
		\end{equation}
		We apply Theorem \ref{thm_ym} for $(E_k, h^{E_k})$, $k \in \nat$, $c \in \real$, to get
		\begin{equation}\label{eq_HYM_k}
			{\rm{HYM}}_{+\infty, ck}(E_k, h^{E_k}) 
			\geq
			\max \Big\{ | \mu_{\min}^k - c k |, | \mu_{\max}^k - c k | \Big\}.
		\end{equation}
		If, moreover, $\omega_B$ is Kähler, then for any $p \in [1, +\infty[$, we have
		\begin{equation}\label{eq_HYM_k2}
			{\rm{HYM}}_{p, c k}(E_k, h^{E_k}) 
			\geq
			N_k
			\cdot
			\int_{\real} |x - c k|^p  d \eta_{k, 0}^{HN}(x)
			\cdot
			\int_B [\omega_B^m].
		\end{equation}
		\begin{sloppypar}
		Directly from Theorem \ref{thm_mazh} and Proposition \ref{prop_toepl}, under the respective assumptions, for any $p \in [1, +\infty[$, $c \in \real$, we have
		\begin{equation}\label{eq_fym_sm_limit}
			\begin{aligned}
			&
			\lim_{k \to \infty} \frac{{\rm{HYM}}_{p, c k}(E_k, h^{E_k})}{k^p \cdot N_k} 
			=
			\frac{{\rm{FYM}}_{p, c}(\pi, h^L)}{\int_{X_b} [\omega^n]},
			\\
			&
			\lim_{k \to \infty} \frac{{\rm{HYM}}_{+ \infty, c k}(E_k, h^{E_k})}{k} 
			=
			{\rm{FYM}}_{+ \infty, c}(\pi, h^L),
			\end{aligned}
		\end{equation}
		where $b \in B$ is an arbitrary point, and $X_b$ is the fiber of $\pi$ at $b$.
		\end{sloppypar}
		We now divide both sides of the first inequality of (\ref{eq_HYM_k}) by $k^p \cdot N_k$, take the limit $k \to \infty$, and apply Theorem \ref{thm_conv_meas} and (\ref{eq_fym_sm_limit}) to deduce
		\begin{equation}
			{\rm{FYM}}_{p, c}(\pi, h^L) 
			\geq
			\int_{\real} \big|x - c \big|^p d \eta^{HN}(x)
			\cdot
			\int_{X_b} [\omega^n] \cdot \int_B [\omega_B^m].
		\end{equation}		
		This establishes Theorem \ref{thm_obs} for $p \in [1, +\infty[$, as $\int_{X_b} [\omega^n] \cdot \int_B [\omega_B^m] = \int_X [\omega^n] \cdot \pi^*[\omega_B^m]$.
		To get Theorem \ref{thm_obs} for $p = +\infty$, we divide both sides of the second inequality of (\ref{eq_HYM_k}) by $k$, take limit $k \to \infty$, and apply Theorem \ref{thm_conv_meas} and (\ref{eq_fym_sm_limit}).
	\end{proof}

	\section{Horizontal curvature on subfamilies and generic fibered nefness}\label{sect_hor_curv}
	The main goal of this section is to establish Theorems \ref{thm_trshld}, \ref{thm_generic_nef}.
	For this, we construct a sequence of metrics on the polarizing line bundle from a sequence of Hermitian metrics on direct images.
	To show that horizontal mean curvature behaves well under this procedure, we rely on a fibered analogue of the principle that “a curvature of a vector bundle increases under taking quotients".	
	\par 
	More precisely, consider a holomorphic submersion $\pi : X \to B$ between compact complex manifolds $X$ and $B$.
	We denote $m := \dim B$.
	Consider an embedding $\iota : Y \hookrightarrow X$ of a smooth complex manifold $Y$, such that restriction of $\pi$, $\pi|_Y : Y \to B$, is a submersion.
	We fix a $(1, 1)$-form $\omega_X$ on $X$, which is positive along the fibers of $\pi$, and denote $\omega_Y := \iota^* \omega_X$.
	We fix a Hermitian $(1, 1)$-form $\omega_B$ on $B$, and denote by $\wedge_{\omega_B} \omega_{Y, H} \in \ccal^{\infty}(Y)$, $\wedge_{\omega_B} \omega_{X, H} \in \ccal^{\infty}(X)$, the horizontal mean curvatures of $\omega_Y$ and $\omega_X$ respectively.
	\begin{lem}\label{lem_hor_curv}
		For any $y \in Y$, we have
			$\wedge_{\omega_B} \omega_{Y, H}(y)
			\geq
			\wedge_{\omega_B} \omega_{X, H}( \iota(y) )$.
	\end{lem}
	\begin{rem}
		An equivalent result was established in \cite[(2.2)]{FengLiuWang} by a slightly different method.
	\end{rem}
	\begin{proof}
		Let us fix $y \in Y$, and denote by $b := \pi|_Y(y)$, $x := \iota(y)$, and by $e_1, \ldots, e_m$ an orthonormal basis of $T^{1, 0}_bB$ with respect to $\omega_B$.
		We denote by $e_1^X, \ldots, e_m^X \in T^{1, 0}_x X$ the horizontal lifts of $e_1, \ldots, e_m$, defined with respect to $\omega_X$, i.e. $d \pi(e_i^X) = e_i$, $i = 1, \ldots, m$ and $e_1^X, \ldots, e_m^X$ are orthogonal (with respect to $\omega_X$) to the tangent space of the fibers, $T^V X$, of $\pi$.
		Similarly, we denote by $e_1^Y, \ldots, e_m^Y \in T^{1, 0}_y Y$ the horizontal lifts of $e_1, \ldots, e_m$, defined with respect to $\omega_Y$.
		Clearly, using implicitly the embedding of $T_y Y$ in $T_x X$ through $\iota$, we can write $e_i^Y = e_i^X + v_i$, where $v_i \in T^V_x X$.
		But then, since $e_i^X$ and $v_i$ are orthogonal with respect to $\omega_X$, and $\omega_X$ is positive in the vertical directions, we obtain
		$\imun \omega_Y(e_i^Y, \overline{e}_i^Y) = \imun \omega_X(e_i^X, \overline{e}_i^X) + \imun \omega_X(v_i, \overline{v}_i) \geq \imun \omega_X(e_i^X, \overline{e}_i^X)$.
		By taking a sum of the above inequality over all $i = 1, \ldots, m$, we establish the needed inequality.
	\end{proof}
	Another ingredient we need is the calculation of the horizontal mean curvature for projectivizations of vector bundles.
	More precisely, let $(F, h^F)$ be a Hermitian vector bundle over $B$ of rank $r$. 
	Let $\mathcal{O}(1)$ be the hyperplane bundle over $\mathbb{P}(F^*)$, $\pi : \mathbb{P}(F^*) \to X$.
	We endow $\mathcal{O}(1)$ with the metric $h^{\mathcal{O}(1)}$ induced by $h^F$.
	We denote by $R^F$ the curvature of the Chern connection on $(F, h^F)$, by $\omega$ the first Chern class of $(\mathcal{O}(1), h^{\mathcal{O}(1)})$, and by $\omega_H$ its horizontal component.
	\begin{lem}\label{lem_low_bnd}
		In the above notations, for any $x \in X$, we have 
		\begin{equation}
			\inf_{y \in \mathbb{P}(F_x^*)} \wedge_{\omega_B} \omega_H(y) = \inf_{ \substack{ f \in F_x, \| f \|_{h^F} = 1}}  \Big\langle \frac{\imun}{2 \pi} \wedge_{\omega_B} R^F f, f \Big\rangle_{h^F}.
		\end{equation}
	\end{lem}
	\begin{proof}
	We fix some local coordinates $z := (z_1, \ldots, z_n)$ on $B$, centered at $x \in B$, and a local normal frame $f_1, \ldots, f_{r}$ of $F$ at $x$, defined in a neighborhood $U$ of $x$. By a \textit{normal frame} we mean a holomorphic frame satisfying $\scal{f_i}{f_j}_{h^F} = \delta_{ij} - \sum_{\lambda \mu} d_{\lambda \mu i j} z_{\lambda} \overline{z}_\mu + O(|z|^3)$ for some constants $d_{\lambda \mu i j}$.
	We denote by $f_1^*, \ldots, f_r^*$ the dual frame of $F^*$. 
	The above data defines a trivialization of $U \times \mathbb{P}(\comp^{r}) \to \mathbb{P}(F^*)$ near $\pi^{-1}(x)$ as follows.
	For $a := (a_1, \ldots, a_{r})$, where $a_i \in \comp$, $1 \leq i \leq r$, and not all $a_i$ are equal to zero, the trivialization is given by the map
	$(z, [a]) 
		\to
		[ \sum_{i = 1}^{r} a_i f_i^*(z) ] \in \mathbb{P}(F^*)$.
	Now we take $a_1 = 1$ and denote $b_i := a_i$, $2 \leq i \leq r$, $b := (b_i)$. Then $(z, b)$ gives a chart for $\mathbb{P}(F^*)$.  
	The well-known formula, cf. \cite[Formula (V.15.15)]{DemCompl}, shows that at the point $(x, [f_1^*]) \in \mathbb{P}(F^*)$, the curvature, $R^{\mathcal{O}(1)}$, of the hyperplane bundle $(\mathcal{O}(1), h^{\mathcal{O}(1)})$, equals
		\begin{equation}\label{eq_mour_form}
			R^{\mathcal{O}(1)}_{(x, [f_1^*])} = \sum_{2 \leq j \leq r} 
			 db_j \wedge d\overline{b}_j
			+
			\scal{R^F f_1}{f_1}_{h^F}.
		\end{equation}
	\par 
	In particular, we see that the vertical part of the form $\omega = c_1(\mathcal{O}(1), h^{\mathcal{O}(1)})$ is the Fubini-Study form induced by $h^F$, and the horizontal part of $\omega$, $\omega_H$, evaluated at $(x, [f_1^*]) \in \mathbb{P}(E^*)$, coincides with $\frac{\imun}{2 \pi} \scal{R^F f_1}{f_1}_{h^F}$.
	The result follows directly from this.
	\end{proof}
	\begin{rem}\label{eq_he_fe_rel}
	From the proof of the above lemma, we see that $\wedge_{\omega_B} \omega_H$ is constant if and only if $\wedge_{\omega_B} R^F$ is the identity endomorphism up to a constant, which means that $\omega$ is fibered Einstein if and only if $(F, h^F)$ is Hermite Einstein.
	\end{rem}
	\par 
	Recall also the following result.
	\begin{prop}[{ \cite[Theorem 1.5]{LiZhaingHN} }]\label{cor_low_bnd_curv}
		For any $\epsilon > 0$, there is a Hermitian metric $h^E_{\epsilon}$ on $E$, such that the associated curvature, $R^{E}_{\epsilon}$, for any $b \in B$, $e \in E_b$, verifies
		$
			\frac{\imun}{2 \pi} \wedge_{\omega_B} R^{E}_{\epsilon} \geq (\mu_{\min} - \epsilon) \cdot {\rm{Id}}_E.
		$
	\end{prop}

	\begin{proof}[Proof of Theorem \ref{thm_trshld}.]
		First of all, for a given $\epsilon > 0$, let $k \in \nat$ be such that $\frac{\mu_{\min}^k}{k} > \eta_{\min}^{HN} - \frac{\epsilon}{2}$.
		By Corollary \ref{cor_low_bnd_curv}, there is a metric $h_k^{\epsilon}$ on $E_k$ such that for the associated curvature, $R^{E_k}_{\epsilon}$, we have
		\begin{equation}\label{eq_appr_eps_proj0}
			\frac{\imun}{2 \pi} \wedge_{\omega_B}  R^{E_k}_{\epsilon} \geq \Big( \mu_{\min}^k - \frac{\epsilon}{2} \Big) \cdot {\rm{Id}}_{E_k}.
		\end{equation}
		We denote by $\omega_k$ the $(1, 1)$-form on $\mathbb{P}(E_k^*)$, given by the first Chern class of the curvature of the hyperplane line bundle induced by the metric $h_k^{\epsilon}$.
		We denote by $\omega_{H, k}$ the horizontal part of this curvature.
		From Lemma \ref{lem_low_bnd}, the choice of $k \in \nat$ and (\ref{eq_appr_eps_proj0}), we deduce
		\begin{equation}\label{eq_appr_eps_proj}
			\inf_{x \in \mathbb{P}(E_k^*)} \wedge_{\omega_B} \omega_{H, k}(x) 
			\geq 
			 k \cdot ( \eta_{\min}^{HN} - \epsilon ).
		\end{equation}
		We will now assume that $k$ was chosen big enough so that $L^k$ is relatively ample.
		Consider now the Kodaira embedding $\iota_k : X \hookrightarrow \mathbb{P}(E_k^*)$.
		It is well-known that there is a canonical isomorphism between $\iota_k^* \mathscr{O}(1)$ and $L^k$.
		We denote by $h^L_{\epsilon}$ the metric induced on $L$ by the pull-back; then $\omega(h^L_{\epsilon}) = \frac{1}{k} \iota_k^* \omega_{H, k}$.
		By Lemma \ref{lem_hor_curv} and (\ref{eq_appr_eps_proj}), we conclude that
		 $\inf_{x \in X} \wedge_{\omega_B} \omega_H(h^L_{\epsilon})
		\geq \eta_{\min}^{HN} - \epsilon$.
		Since $\epsilon > 0$ can be taken arbitrarily small, we deduce $\sup_{h^L} \inf_{x \in X} \wedge_{\omega_B} \omega_{H}(h^L)
		\geq \eta_{\min}^{HN}$. 
		In combination with the upper bound from (\ref{eq_ess_sup_inf1}), this finishes the proof.
	\end{proof}
	Now, let us establish Theorem \ref{thm_generic_nef}.
	As in the statement of Theorem \ref{thm_generic_nef}, we fix any Kähler forms $\omega_{B, 1}, \ldots, \omega_{B, m - 1}$, $\omega_X$.
	Then the form $\omega_{B, 1} \wedge \cdots \wedge \omega_{B, m - 1}$ is positive in the sense of \cite[(III.1.1)]{DemCompl}.
	By Michelsohn \cite[(4.8)]{MichelStrPos}, there is a Hermitian form $\omega_B$ on $B$, verifying 
	\begin{equation}\label{eq_multipol_gaud}
		\omega_B^{m - 1} = \omega_{B, 1} \wedge \cdots \wedge \omega_{B, m - 1}.
	\end{equation}
	Remark that $\omega_B$ is automatically Gauduchon, but not necessarily Kähler. 
	Since the form $\omega_B$ is Gauduchon, it makes sense to define the $\omega_B$-degree and study the Harder-Narasimhan $\omega_B$-slopes, as we did before Theorem \ref{thm_singleton}.
	Remark that due to a relation (\ref{eq_multipol_gaud}), this $\omega_B$-degree coincides with the degree associated with a multipolarization $([\omega_{B, 1}], \ldots, [\omega_{B, m - 1}])$, as defined in \cite[before (1.1)]{FinHNI}.
	Below, the invariant $\eta_{\min}^{HN}$ and other quantities are calculated with respect to $\omega_B$.
	The following result, alongside with Theorem \ref{thm_trshld}, lie at the core of the proof of Theorem \ref{thm_generic_nef}.
	\begin{sloppypar}	
	\begin{prop}[{\cite[Proposition 5.2]{FinHNI}}]\label{prop_interpr_st_gen_pos}
		A relatively ample line bundle $L$ over $X$ is stably $([\omega_{B, 1}], \ldots, [\omega_{B, m - 1}])$-generically fibered nef with respect to $\pi$ if and only if $\eta_{\min}^{HN} \geq 0$.
	\end{prop}
	\begin{proof}[Proof of Theorem \ref{thm_generic_nef}]
		Let us first establish Theorem \ref{thm_generic_nef} under an additional assumption that $L$ is relatively ample.
		We assume first that $L$ is stably $([\omega_{B, 1}], \ldots, [\omega_{B, m - 1}])$-generically fibered nef with respect to $\pi$.
		By Theorem \ref{thm_trshld} and Proposition \ref{prop_interpr_st_gen_pos}, we establish that for any $\epsilon > 0$, there is a relatively positive Hermitian metric $h^L_{\epsilon}$ on $L$, such that $\wedge_{\omega_B} \omega_H(h^L_{\epsilon}) > - \epsilon$.
		From the definition of $\wedge_{\omega_B}$ and the trivial fact that there is $C > 0$, such that $\omega_B^m < C \omega_X^m$, we establish 
		\begin{equation}\label{eq_generic_nef101}
			\omega_H(h^L_{\epsilon}) \wedge \pi^* \omega_{B, 1} \wedge \cdots \wedge \pi^* \omega_{B, m - 1}  
			\geq
			-C \epsilon \cdot \omega_X^m.
		\end{equation}
		But the form $\omega(h^L_{\epsilon})$ is relatively positive, so (\ref{eq_generic_nef101}) implies (\ref{eq_generic_nef}) for $\epsilon := C \epsilon$, which finishes the proof of one direction of Theorem \ref{thm_generic_nef} under an additional assumption that $L$ is relatively ample.
		\par 
		To prove the opposite direction under the same additional assumption that $L$ is relatively ample, assume that we have a sequence of metrics $h^L_{\epsilon}$, verifying (\ref{eq_generic_nef}).
		We will now show that one can cook up a sequence of relatively positive metrics, $h^L_{\epsilon, 0}$, verifying similar bounds.
		Indeed, let us fix an arbitrary relatively positive metric $h^L_0$ on $L$. 
		By (\ref{eq_generic_nef}), it is easy to see that there is $c > 0$, such that for any $\epsilon > 0$, over the fibers, the following inequality is satisfied $c_1(L, h^L_{\epsilon}) \geq -\epsilon c \cdot c_1(L, h^L_0)$.
		Then an easy calculation shows that the sequence of metrics $h^L_{\epsilon, 0} := (h^L_{\epsilon})^{1 - 2c \epsilon} \cdot (h^L_0)^{2 c \epsilon}$ is positive along the fibers and verifies the inequality (\ref{eq_generic_nef}) with $C \epsilon$ in place of $\epsilon$, for some $C > 0$.
		Then as in (\ref{eq_generic_nef101}), there is $C > 0$, such that for any $\epsilon > 0$, we have $\wedge_{\omega_B} \omega_H(h^L_{\epsilon, 0}) > - C \epsilon$.
		By Theorem \ref{thm_trshld}, we then conclude that $\eta_{\min}^{HN} \geq 0$, which implies that $L$ is stably $([\omega_{B, 1}], \ldots, [\omega_{B, m - 1}])$-generically fibered nef by Proposition \ref{prop_interpr_st_gen_pos}.
		\par 
		We now only assume that $L$ is relatively nef.
		We assume first that $L$ is stably $([\omega_{B, 1}], \ldots, [\omega_{B, m - 1}])$-generically fibered nef with respect to $\pi$.
		Now, for any $\delta \in \mathbb{Q}$, $\delta > 0$, consider the $\mathbb{Q}$-line bundle $L_{\delta} := L \otimes L_0^{\delta}$, where $L_0$ is some ample line bundle on $X$.
		Clearly, $L_{\delta}$ is relatively ample, and it is also stably $([\omega_{B, 1}], \ldots, [\omega_{B, m - 1}])$-generically fibered nef with respect to $\pi$. 
		The already established relatively ample case of Theorem \ref{thm_generic_nef} says that $L_{\delta}$ is $([\omega_{B, 1}], \ldots, [\omega_{B, m - 1}])$-generically fibered nef if and only if for any $\epsilon > 0$, there is a Hermitian metric $h^{L_{\delta}}_{\epsilon}$ on $L_{\delta}$, such that the analogue of (\ref{eq_generic_nef}) holds. 
		Let $h^L_0$ be now an arbitrary positive metric on $L_0$. 
		It is easy to see that if $\delta$ and $\epsilon$ are sufficiently small, then the metric $h^L_{\epsilon}$ on $L$, which is constructed as the only metric verifying $h^{L_{\delta}}_{\epsilon} = h^L_{\epsilon} \cdot (h^L_0)^{\delta}$, will satisfy the analogue of (\ref{eq_generic_nef}) (for $C \epsilon$ instead of $\epsilon$ for some $C > 0$).
		This shows one direction of Theorem \ref{thm_generic_nef}.
		\par 
		Inversely, if for any $\epsilon > 0$ there is a metric $h^L_{\epsilon}$ as in (\ref{eq_generic_nef}), then the metrics $h^{L_{\delta}}_{\epsilon}$, defined by the above formula, will also satisfy a similar inequality.
		Hence, by the already established case of Theorem \ref{thm_generic_nef}, $L_{\delta}$ is then stably $([\omega_{B, 1}], \ldots, [\omega_{B, m - 1}])$-generically fibered nef for any $\delta \in \mathbb{Q}$, $\delta > 0$.
		In particular, for any $\delta \in \mathbb{Q}$, $\delta > 0$, the line bundle $L_{2 \delta} = L_{\delta} \otimes L_0^{\delta}$ is $([\omega_{B, 1}], \ldots, [\omega_{B, m - 1}])$-generically fibered nef for any $\delta \in \mathbb{Q}$, $\delta > 0$, which means that $L$ is stably $([\omega_{B, 1}], \ldots, [\omega_{B, m - 1}])$-generically fibered nef, as $L_{2 \delta} = L \otimes L_0^{2 \delta}$. This finishes the proof.
	\end{proof}
	\end{sloppypar}

	\section{Asymptotic Riemann-Roch-Grothendieck and semistability}\label{sect_num_obs}
	The main goal of this section is to prove a numerical obstruction for asymptotic semistability of direct images from Theorem \ref{thm_nonlinstab}.
	This will be based on an asymptotic version of Riemann-Roch-Grothendieck theorem, which we establish here in the singular setting.
	\par 
	To begin, let us recall some basic facts about Bott-Chern and Aeppli cohomologies.
	Let $Y$ be a compact complex manifold of dimension $n$.
	We denote by $\Omega^{(p, q)}(Y)$ the vector space of $(p, q)$-differential forms on $Y$, $p, q \in \nat$, and define $\partial: \Omega^{(p, q)}(Y) \to \Omega^{(p + 1, q)}(Y)$, $\dbar: \Omega^{(p, q)}(Y) \to \Omega^{(p, q + 1)}(Y)$, as usual. 
	Recall that Bott-Chern cohomology, $H^{p, q}_{BC}(Y)$, is defined as
	\begin{equation}
		H^{p, q}_{BC}(Y)
		:=
		\frac{(\ker \partial: \Omega^{(p, q)}(Y) \to \Omega^{(p + 1, q)}(Y)) \cap (\ker \dbar: \Omega^{(p, q)}(Y) \to \Omega^{(p, q + 1)}(Y))}{{\rm{im}} \partial \dbar: \Omega^{(p - 1, q - 1)}(Y) \to \Omega^{(p, q)}(Y)}.
	\end{equation}
	Recall that Aeppli cohomology, $H^{p, q}_{A}(Y)$, is defined as
	\begin{equation}
		H^{p, q}_{A}(Y)
		:=
		\frac{\ker \partial \dbar : \Omega^{(p, q)}(Y) \to \Omega^{(p + 1, q + 1)}(Y)}{({\rm{im}} \partial: \Omega^{(p - 1, q)}(Y) \to \Omega^{(p, q)}(Y)) + ({\rm{im}} \dbar: \Omega^{(p, q - 1)}(Y) \to \Omega^{(p, q)}(Y))}.
	\end{equation}
	It is standard that for compact Kähler manifolds, the two cohomologies coincide.
	For $p, q = 0, \ldots, n$, we have the natural pairing 
	\begin{equation}\label{eq_int_pairing}
		\wedge : H^{p, q}_{BC}(Y) \times H^{n - p, n - q}_{A}(Y) \to \comp,
	\end{equation}
	given by the wedge product and integration.
	If $p: Y \to B$ is a holomorphic map between compact complex manifolds, then for $s := \dim Y - \dim B$, we have a natural map $p_* : H^{p, q}_{BC}(Y) \to H^{p - s, q - s}_{BC}(B)$, defined by the pairing (\ref{eq_int_pairing}) and the pull-back $p^*$.
	\par 
	The Bott-Chern cohomology can be generally defined for arbitrary complex analytic spaces $Y$ in the sense of \cite[Definition II.5.2]{DemCompl}, where one considers differential forms on $Y$ obtained by pullbacks of smooth differential forms through local embeddings of the space into complex vector spaces.
	Due to a theorem of Lelong, cf. \cite[Theorem III.2.7]{DemCompl}, the intersection pairing (\ref{eq_int_pairing}) can still be defined in this setting, and so the slope (\ref{eq_defn_slope_subm}) is well-defined.
	By resolving the singularities, we can extend the definition of the pushforward for maps between a compact complex analytic spaces $Y$ and a compact manifold $B$.
	\par 
	Now, let $E$ be a holomorphic vector bundle over $Y$.
	Using Chern-Weil theory, one can construct for any Hermitian metric $h^E$ on $E$ a corresponding Chern character form, $\ch(E, h^E)$, which is a $d$-closed form in $\oplus_{p = 0}^{+ \infty} \Omega^{(p, p)}(Y)$.
	Bott-Chern in \cite{BottChern} showed that the resulting class in Bott-Chern cohomology doesn't depend on the choice of the metric.
	This gives a definition of the Chern character, $\ch(E)$ of a vector bundle $E$ with values in Bott-Chern cohomology.
	\par 
	Bismut-Shu-Wei in \cite{BismutShuWei} generalized the definition of the Chern character with values in Bott-Chern cohomology for any coherent sheaf $\mathscr{E}$ on $Y$.
	If $\mathscr{E}$ has a finite locally free projective resolution (which is always the case if $Y$ is projective), this construction corresponds to the one given by the alternating sum of Chern characters of the resolution.
	By \cite[\S 8.6]{BismutShuWei}, the $(k, k)$-component of the Chern character, $\ch_k(\mathscr{E})$, we have
	\begin{equation}\label{eq_ch01}
		\ch_0(\mathscr{E}) = \rk{\mathscr{E}}, \qquad \ch_1(\mathscr{E}) = c_1(\det \mathscr{E}),
	\end{equation}
	where $\det \mathscr{E}$ is the Knudsen-Mumford determinant \cite{Knudsen1976}.
	Without entering into details of the construction, we mention that the absence of finite locally free projective resolutions of coherent sheaves for general complex manifolds is circumvented in \cite{BismutShuWei} by the use of so-called antiholomorphic superconnections, see \cite[Theorem 6.7]{BismutShuWei}.
	\par 
	Now, recall that for any proper holomorphic map $p : Y \to B$, and any coherent sheaf $\mathscr{E}$, Grauert theorem tells that the direct image sheaves $R^q p_* \mathscr{E}$, $q \in \nat$, are coherent.
	The main result of this section goes as follows.
	\begin{thm}\label{thm_rrg_asym}
		Let $Y$ be an irreducible compact complex analytic space, $L$ an arbitrary line bundle on $Y$ and $\mathscr{E}$ a coherent sheaf on $Y$.
		Let $p : Y \to B$ be a holomorphic map to a compact complex manifold $B$.
		Then for any $r \in \nat$, and $s := \dim Y - \dim B + r$, in the Bott-Chern cohomology
		\begin{equation}
			\lim_{k \to \infty} \frac{1}{k^s}
			\sum_{t = 0}^{\dim Y} (-1)^t \ch_r(R^t p_*(\mathscr{E} \otimes L^k) )
			=
			\frac{\rk{\mathscr{E}}}{s!} \cdot p_*(c_1(L)^s).
		\end{equation}
	\end{thm}
	\begin{rem}
		Despite a huge amount of literature, we were not able to find the proof of this result under the stated hypothesizes (even for projective $Y, B$).
		For flat maps $p$ and relatively ample $L$, this result can be alternatively established using Knudsen-Mumford expansions, see \cite{Knudsen1976} and Phong-Ross-Sturm \cite[Theorem 3]{PhongRossSturm}.
		Note, however, that flatness doesn't pass through subfamilies, and there is no flatness	 assumption in Theorem \ref{thm_nonlinstab}.
	\end{rem}
	The proof of this result relies on the recent result of Bismut-Shu-Wei \cite{BismutShuWei} establishing the Riemann-Roch-Grothendieck theorem in Bott-Chern cohomology for arbitrary holomorphic maps between (\textit{smooth}!) complex manifolds.
	More precisely, the main result of \cite{BismutShuWei} says the following.
	\begin{thm}\label{thm_rrg}
		Let $p : Y \to B$ be a holomorphic map between compact complex manifolds.
		Then for any coherent sheaf $\mathscr{E}$ on $Y$, the following identity holds in Bott-Chern cohomology
		\begin{equation}
			\td(TB) \cdot \sum_{t = 0}^{\dim Y} (-1)^t \ch(R^t p_*(\mathscr{E}) )
			=
			p_* (
			\td(TY) \cdot \ch( \mathscr{E} )
			),
		\end{equation}
		where $\td(TB)$, $\td(TY)$ are the Todd classes.  
	\end{thm}
	\begin{proof}[Proof of Theorem \ref{thm_rrg_asym}]
		Remark first that for smooth manifolds $Y$, the result follows directly from Theorem \ref{thm_rrg} by (\ref{eq_ch01}) and the fact that the $0$-degree part of the Todd class is identity.
		If $\dim Y = 0$, then $Y$ is automatically smooth, and, hence, Theorem \ref{thm_rrg_asym} holds as stated.
		\par 
		We will argue by induction on the dimension of $Y$. 
		For this, we consider a resolution of singularities $f : \hat{Y} \to Y$ of $Y$.
		For any $q = 1, \ldots, \dim Y$, we define the sheaves $\mathscr{Q}_q$ on $Y$ as $\mathscr{Q}_q := R^q f_*  f^* \mathscr{E}$, where $f^* \mathscr{E} := f^{-1} \mathscr{E} \otimes_{f^{-1} \mathscr{O}_{Y}} \mathscr{O}_{\hat{Y}}$.
		We define the sheaf $\mathscr{Q}_0$ on $Y$ by the following short exact sequence 
		\begin{equation}\label{eq_q0_defn}
			0 
			\rightarrow
			\mathscr{E}
			\rightarrow
			R^0 f_* f^* \mathscr{E}
			\rightarrow
			\mathscr{Q}_0
			\rightarrow
			0.
		\end{equation}
		By Grauert theorem, the sheaves $\mathscr{Q}_q$, $q = 0, \ldots, \dim Y$, are coherent.
		Since the resolution of singularities is biholomorphic away from a subset of singular points of $Y$, and over the locally free locus of $\mathscr{E}$, by projection formula, cf. \cite[Exercise II.5.1d)]{HartsBook}, we have $R^0 f_* f^* \mathscr{E} = \mathscr{E}$, the supports of the sheafs $\mathscr{Q}_q$, $q = 0, \ldots, \dim Y$, are \textit{proper} analytic subsets of $Y$, which, by irreducibility of $Y$, have strictly smaller dimension than $Y$, cf. \cite[Proposition II.4.2.6]{DemCompl}.
		By this and the usual \textit{devissage} techniques, see \cite[Théorème 3.1.2]{EGA3}, cf. \cite[Proposition I.7.4]{HartsBook}, for any $q = 0, \ldots, \dim Y$, there is $r(q) \in \nat$, and complex analytic subspaces $\iota_{i, q}: Z_{i, q} \hookrightarrow Y$, with some ideal sheaves $\mathscr{J}_{i, q}$ on $Z_{i, q}$ and a filtration $\mathscr{F}_{i, q}$ of $\mathscr{Q}_q$, $i = 0, \ldots, r(q)$, $\mathscr{F}_{0, q} = \{0\}$, $\mathscr{F}_{r(q), q} = \mathscr{Q}_q$, $\mathscr{F}_{i - 1, q} \subset \mathscr{F}_{i, q}$, $i = 1, \ldots, r_q$, such that for any $i = 1, \ldots, r(q)$, we have $\mathscr{F}_{i, q} / \mathscr{F}_{i - 1, q} = \iota_{i, q, *}(\mathscr{J}_{i, q})$.
		We denote $p_{i, q} := p \circ \iota_{i, q}$, $i = 1, \ldots, r(q)$, $\hat{p} := p \circ f$. 
		We argue that for any $k \in \nat$, we have
		\begin{multline}\label{eq_ident_resol}
			\sum_{t = 0}^{\dim Y} (-1)^t
			\ch(R^t p_*(\mathscr{E} \otimes L^k) )
			=
			\sum_{t = 0}^{\dim Y} (-1)^t
			\ch(R^t \hat{p}_*(f^* \mathscr{E} \otimes f^* L^k) )
			\\
			-
			\sum_{t, u = 0}^{\dim Y} (-1)^{t + u} \sum_{i = 1}^{r(q)}
			\ch(R^t (p_{i, u})_*(\mathscr{J}_{i, u} \otimes \iota_{i, u}^* L^k ) ).
		\end{multline}
		Once (\ref{eq_ident_resol}) is established, Theorem \ref{thm_rrg_asym} would follow by induction, as the space $\hat{Y}$ is smooth, and so for the first summand on the right-hand side of (\ref{eq_ident_resol}), the smooth version of Theorem \ref{thm_rrg_asym} applies, and the second summand doesn't contribute to the asymptotics by induction hypothesis, as all $Z_{i, q}$ have strictly smaller dimensions than $Y$.
		\par 
		Now, let us establish (\ref{eq_ident_resol}). First of all, since in derived category, there is a canonical isomorphism between the functors $R \hat{p}_*$ and $R p_* R f_*$, cf. \cite[(3.13)]{BismutShuWei}, and the construction of Chern character, defined using derived category of coherent sheaves, factors through the $K$-theory of the derived category of coherent sheaves \cite[Theorem 8.11 and \S 8.9]{BismutShuWei}, we have 
		\begin{equation}\label{eq_der_desing}
			\sum_{t = 0}^{\dim Y} (-1)^t
			\ch(R^t \hat{p}_*(f^* \mathscr{E} \otimes f^*L^k) )
			=
			\sum_{t, u = 0}^{\dim Y} (-1)^{t + u}
			\ch(R^t p_*( R^u f_*( f^* \mathscr{E}  \otimes f^*L^k) ).
		\end{equation}
		Now, from the exact sequence (\ref{eq_q0_defn}), using again the fact that the construction of the Chern character passes through the formation of $K$-theory, we obtain 
		\begin{multline}\label{eq_q1_inc1}
			\sum_{t = 0}^{\dim Y} (-1)^{t}
			\ch(R^t p_*( R^0 f_*( f^* \mathscr{E} \otimes f^*L^k) )
			\\
			=
			\sum_{t = 0}^{\dim Y} (-1)^{t}
			\ch(R^t p_* (\mathscr{E} \otimes L^k) )
			+
			\sum_{t = 0}^{\dim Y} (-1)^{t}
			\ch(R^t p_*( \mathscr{Q}_0 \otimes L^k )).
		\end{multline}
		Similarly, for any $q = 0, \ldots, \dim Y$, we obtain 
		\begin{equation}\label{eq_q1_inc2}
			\sum_{t = 0}^{\dim Y} (-1)^t
			\ch(R^t p_*( \mathscr{Q}_q \otimes L^k ))
			=
			\sum_{t = 0}^{\dim Y} (-1)^t
			\sum_{i = 0}^{r(q)}
			\ch(R^t p_*(\iota_{i, q, *} (\mathscr{J}_{i, q}) \otimes L^k )).
		\end{equation}
		Remark, however, that since $\iota_{i, q}$ is a closed embedding, $\iota_{i, q, *}$ is an exact functor, so we have $R^v \iota_{i, q, *} = 0$ for $v = 1, \ldots, \dim Y$, and $R^0 \iota_{i, q, *} = \iota_{i, q, *}$.
		Moreover, by the projection formula, cf. \cite[Exercise II.5.1d)]{HartsBook}, we have $R^0 \iota_{i, q, *} (\mathscr{J}_{i, q}) \otimes L^k  = R^0 \iota_{i, q, *} (\mathscr{J}_{i, q} \otimes \iota_{i, q}^* L^k)$.
		In particular, for any $t, q = 0, \ldots, \dim Y$, $i = 1, \ldots, r(q)$, we can write
		\begin{equation}\label{eq_q1_inc22}
			\ch(R^t p_*(\iota_{i, q, *} (\mathscr{J}_{i, q}) \otimes L^k ))
			=
			\sum_{v = 0}^{\dim Y} (-1)^v
			\ch(R^t p_*( R^v \iota_{i, q, *} (\mathscr{J}_{i, q} \otimes \iota_{i, q}^* L^k) )).
		\end{equation}
		But using the same argument as in (\ref{eq_der_desing}), we have
		\begin{equation}\label{eq_q1_inc222}
			\sum_{t, v = 0}^{\dim Y} (-1)^{t + v}
			\ch(R^t p_*( R^v \iota_{i, q, *} (\mathscr{J}_{i, q} \otimes \iota_{i, q}^* L^k) ))
			=
			\sum_{t = 0}^{\dim Y} (-1)^{t}
			\ch(R^t (p_{i, q})_*(\mathscr{J}_{i, q} \otimes \iota_{i, q}^* L^k ) ).
		\end{equation}
		Now, a combination of (\ref{eq_der_desing}), (\ref{eq_q1_inc1}), (\ref{eq_q1_inc2}), (\ref{eq_q1_inc22}) and (\ref{eq_q1_inc222}), gives us (\ref{eq_ident_resol}).
	\end{proof}
	Now, let us finally establish an application of Theorem \ref{thm_rrg_asym} towards the study of Harder-Narasimhan slopes of direct images. 
	We fix an irreducible compact complex analytic space $Y$ of dimension $k + m$, $k \geq 0$, with a surjective holomorphic map $\pi : Y \to B$.
	Let $\omega_B$ be a Gauduchon Hermitian form on $B$.
	Let $L$ be a relatively ample line bundle on $Y$.
	Recall that in (\ref{eq_defn_slope_subm}), we defined the slope of $Y$, and before Theorem \ref{thm_singleton}, we defined the slopes of coherent sheafs.
	\begin{lem}\label{lem_asympt_slope_calc}
		The following identity holds
		$
			\mu(Y) \cdot \int_B [\omega_B^m]
			=
			\lim_{k \to \infty} \mu(R^0 \pi_* L^k) / k.
		$
	\end{lem}
	\begin{proof}
		By Serre vanishing theorem, the higher direct images $R^v \pi_* L^k$, $v \geq 1$, vanish.
		By (\ref{eq_ch01}), we conclude that 
		\begin{equation}
			\mu(R^0 \pi_* L^k) 
			=
			\frac{\int_Y [\ch_1(R^0 \pi_* L^k)] \cdot [\omega_B^{m-1}]}{\ch_0(R^0 \pi_* L^k)}.
		\end{equation}
		The result now follows directly from Theorem \ref{thm_rrg_asym}.
	\end{proof}
	\begin{prop}\label{prop_asymp_semist}
		The vector bundles $E_k$ are asymptotically semistable if and only if $\eta_{\min}^{HN} = \eta_{\max}^{HN}$.
		Moreover, if $E_k$ are asymptotically semistable, then for any subsheaves $\mathcal{F}_k$ of $E_k$, $\rk{\mathcal{F}_k} > 0$, and any $\epsilon > 0$, for $k$ big enough, we have $\mu(\mathcal{F}_k) \leq \mu(E_k) + \epsilon k$.
	\end{prop}
	\begin{proof}
		Remark first that the maximal and the minimal slopes satisfy
		\begin{equation}\label{eq_max_min_slope}
			\begin{aligned}
				&
				\mu_{\max}^k
				=
				\sup \Big\{ 
					\mu(\mathcal{F}_k)
					:
					\mathcal{F}_k
					\text{ is a subsheaf of } E_k
				\Big\}
				,
				\\
				&
				\mu_{\min}^k
				=
				\inf \Big\{ 
					\mu(\mathcal{Q}_k)
					:
					\mathcal{Q}_k
					\text{ is a quotient sheaf of } E_k
				\Big\}.
			\end{aligned}
		\end{equation}
		Now, by Theorem \ref{thm_conv_meas}, as $k \to \infty$, we have 
		\begin{equation}\label{eq_mu_ek_asymp}
			\lim_{k \to \infty} \frac{\mu(E_k)}{k} = \int_{\real} x d \eta^{HN}(x).
		\end{equation}
		From (\ref{eq_max_min_slope}) and (\ref{eq_mu_ek_asymp}), we see that $E_k$ are asymptotically semistable if and only if $\eta_{\min}^{HN} = \esssup \eta^{HN}$.
		However, by Theorem \ref{thm_conv_meas}, $\esssup \eta^{HN}$ coincides with $\eta_{\max}^{HN}$, which finishes the proof of the first part of the theorem.
		The proof of the second statement of Proposition \ref{prop_asymp_semist} follows directly by (\ref{eq_max_min_slope}) and the first part.
	\end{proof}
	\begin{rem}
		From \cite[Proposition 5.1]{FinHNI}, cf. also \cite{XuZhuPosCM}, we know that if $\dim B = 1$, then $\essinf \eta^{HN}$ coincides with $\eta_{\min}^{HN}$.
		The above proof shows that for $\dim B = 1$, the condition on the subsheaves from Proposition \ref{prop_asymp_semist} is equivalent to asymptotic semistability.
	\end{rem}
	\begin{proof}[Proof of Theorem \ref{thm_singleton}]
		Follows immediately from Theorem \ref{thm_obs} and Proposition \ref{prop_asymp_semist}.
	\end{proof}

	\begin{prop}\label{prop_mu_min_bnd}
		For any complex analytic subspace $Y$ of $X$ as in Theorem \ref{thm_nonlinstab}, the following bound holds
		$\mu(Y) \cdot \int_B [\omega_B^m] \geq \eta_{\min}^{HN}$. 
	\end{prop}
	\begin{proof}
		Consider the following short exact sequence of sheaves associated with $Y$
		\begin{equation}\label{eq_str_sh_y}
			0 
			\rightarrow
			\mathscr{J}_Y
			\rightarrow
			\mathscr{O}_X
			\rightarrow
			\iota_* \mathscr{O}_Y
			\rightarrow
			0,
		\end{equation}	 
		where $\mathscr{J}_Y$ is the ideal sheaf of $Y$, consisting of local holomorphic functions on $X$, vanishing along $Y$, and $\mathscr{O}_Y$ is the structure sheaf of $Y$ associated with the reduced scheme structure of $Y$, i.e. defined by (\ref{eq_str_sh_y}).
		By considering a long exact sequence of direct images associated with (\ref{eq_str_sh_y}) and the map $\pi$, and using Serre vanishing theorem, we conclude that the restriction map $R^0 \pi_* L^k \to R^0 \pi|_{Y, *} L|_Y^k$ is surjective. 
		Then, in the notations of (\ref{eq_max_min_slope}), we have $\mu_{\min}^k \leq \mu(R^0 \pi|_{Y, *} L|_Y^k)$.
		We deduce Proposition \ref{prop_mu_min_bnd} from Lemma \ref{lem_asympt_slope_calc} by diving by $k$ and passing to the limit $k \to \infty$.
	\end{proof}
	\begin{proof}[Proof of Theorem \ref{thm_nonlinstab}]
		First of all, by Lemma \ref{lem_asympt_slope_calc}, we conclude that 
		\begin{equation}\label{eq_calc_mu_x_tot}
			\lim_{k \to \infty} \mu(E_k)/k = \mu(X) \cdot \int_B [\omega_B^m]  \geq \eta_{\min}^{HN}.
		\end{equation}
		Moreover, since $\esssup \eta^{HN}$ coincides with $\eta_{\max}^{HN}$ by \cite[Theorem 1.1]{FinHNI}, by (\ref{eq_mu_ek_asymp}), we conclude that the equality in the above inequality holds if and only if $\eta_{\min}^{HN} = \eta_{\max}^{HN}$, i.e. when $E_k$ is asymptotically semistable by Proposition \ref{prop_asymp_semist}.
		In particular, if $E_k$ is asymptotically semistable, then by $\eta_{\min}^{HN} = \eta_{\max}^{HN}$, Proposition \ref{prop_mu_min_bnd} and (\ref{eq_calc_mu_x_tot}), we establish the first part of Theorem \ref{thm_nonlinstab}.
		\par 
%		\begin{sloppypar}
		Let us now establish the second part.
		By a reformulation of the result of Xu-Zhuang \cite[Lemma 2.26 and Proposition 2.28]{XuZhuPosCM} from \cite[(1.5)]{FinHNI}, we have
		\begin{equation}\label{eq_num_form}
			\eta_{\min}^{HN}
			=
			\inf_{C \subset X} \mu(C)  \cdot \int_B [\omega_B^m],
		\end{equation}
		where $C$ runs over all irreducible curves in $X$, with project surjectively to $B$.
		In particular, from (\ref{eq_num_form}), we conclude that
		\begin{equation}\label{eq_calc_mu_x_tot1}
			\eta_{\min}^{HN}
			\geq
			\inf_{Y \subset X} \mu(Y)  \cdot \int_B [\omega_B^m],
		\end{equation}
		where $Y$ are as in Theorem \ref{thm_nonlinstab}.
		A combination of (\ref{eq_calc_mu_x_tot}) and (\ref{eq_calc_mu_x_tot1}) shows that if $\inf_{Y \subset X} \mu(Y) = \mu(X)$, then both (\ref{eq_calc_mu_x_tot}) and (\ref{eq_calc_mu_x_tot1}) are actually equalities.
		By the remark after (\ref{eq_calc_mu_x_tot}), this implies that $E_k$ is asymptotically semistable, which finishes the proof.
%		\end{sloppypar}
	\end{proof}
	\begin{rem}
		It is interesting to know if the second part of Theorem \ref{thm_nonlinstab} continues to hold for $\dim B > 1$.
		There are several potential pitfalls for that.
		First, since $B$ is not necessarily projective, there might be very few analytic subspaces $Y \subset X$, projecting surjectively to $B$.
		Second, even for projective $B$, the main result of \cite[Theorem 1.4]{FinHNI} implies that the analogue of the bound (\ref{eq_num_form}) becomes tight if one considers among $Y$ all subcurves in $X$ projecting to generic curves over the base.
		But it seems that there is much more curves like that than analytic subspaces projecting surjectively to the base.
	\end{rem}

\bibliography{bibliography}

		\bibliographystyle{abbrv}

\Addresses

\end{document}